\documentclass[a4paper]{article}

\usepackage{
amsmath,
amsthm,
amscd,
amssymb,
}
\usepackage{comment}
\usepackage{tikz}
\usetikzlibrary{positioning,arrows,calc}
\usepackage{xspace}

\usepackage{graphicx}

\usepackage{csquotes}

\usepackage{url}
\theoremstyle{plain}

\newtheorem{theorem}{Theorem}

\newtheoremstyle{derp}% <name>
{3pt}% <Space above>
{3pt}% <Space below>
{}% <Body font>
{}% <Indent amount>
{\upshape}% <Theorem head font>
{:}% <Punctuation after theorem head>
{.5em}% <Space after theorem headi>
{}% <Theorem head spec (can be left empty, meaning `normal')>
\theoremstyle{derp}

\newcommand{\Z}{\mathbb{Z}}

\newcommand{\ID}{\mathrm{id}}

\newcommand{\RCA}{\mathrm{RCA}}

\newcommand{\Alt}{\mathrm{Alt}}
\newcommand{\ctrl}[2]{\left.{#1}\right\vert_{#2}}

\title{Universal CA groups with few generators}

\author{
Ville Salo \\
vosalo@utu.fi
}

\begin{document}
\maketitle

\begin{abstract}
There exist f.g.-universal cellular automata groups which are quotients of $\Z * \Z_2$ or $\Z_2 * \Z_2 * \Z_2$, as previously conjectured by the author.
\end{abstract}

The following was stated in \cite{Sa18a}: {\textquote{We conjecture that three involutions can generate an f.g.-universal group of RCA.}} We confirm this, and also minimize the size of generating sets for f.g.-universal cellular automata groups. % The main technical tools are the ones from \cite{Sa18a} and the universal gate from \cite{Sa18c}.

The group $\RCA(m)$ is the group of self-homeomorphisms $f$ of $\{0,1,..,m-1\}^\Z$ satisfying $f \circ \sigma = \sigma \circ f$, where $\sigma(x)_i = x_{i+1}$ is the left shift.

\begin{theorem}
Let $G' \in \{\Z * \Z_2, \Z_2 * \Z_2 * \Z_2\}$ and let $m, n \geq 2$ be arbitary. There is a homomorphism $\phi : G' \to \RCA(m)$ such that $\phi(G')$ contains an embedded copy of every finitely-generated group of $\RCA(n)$.
\end{theorem}

\begin{proof}
First consider $G' = \Z * \Z_2$. To show this for all $m, n \geq 2$, it suffices to show it for some $m, n \geq 2$, by \cite{KiRo90}. We let $B$ with $|B| \geq 2$ be arbitrary and $C = \{0, 1\}$ and use the alphabet $A = B \times C$, with $B^\Z$ the ``top track'' and $C^\Z$ the ``bottom track''. By \cite{Sa18a}, there exists a finitely-generated group $H$ of cellular automata containing a copy of every finitely-generated group of cellular automata. By Lemma~7 in \cite{Sa18a} (more precisely, its proof), for any large enough $\ell$ and unbordered word $|w| = \ell$, if a group $G \leq \RCA(B \times C)$ contains
\[ \ctrl{\pi}{[w]_i} \mbox{ and } \ctrl{\pi}{[ww]_i} \]
for all $\pi \in \Alt(\{0,1\}^\ell)$ and all $i \in \Z$, then $G$ contains a copy of $H$. The notation $\ctrl{\pi}{[u]_i}$ is as in Definition~2 of \cite{Sa18a}, and means that we apply $\pi$ on the second track if and only if $u$ appears on the first track, with offset $i$.

Now, let $w \in B^\ell$ be unbordered where $\ell$ is as above, and very large. We construct a $2$-generated group $G$ containing the maps $\ctrl{\pi}{[w]_i}$ and $\ctrl{\pi}{[ww]_i}$, such that one of our generators is an involution.

Let $F : \{0,1\}^n \to \{0,1\}^n$ be a function such that $F^2 = \ID|_{\{0,1\}^n}$ and defining $f : \{0,1\}^\Z \to \{0,1\}^\Z$ by $f(x.wy) = x.F(w)y$, the maps $\sigma^i \circ f \circ \sigma^{-i}$ generate the group of all self-homeomorphisms $g$ of $\{0,1\}^\Z$ for which there exists $m$ such that
\[ \forall x \in \{0,1\}^\Z: \forall |i| \geq m: g(x)_i = x_i \]
holds. Such $F$ exists \cite{Sa18c}.

Our generators are the partial shift on the first track, i.e. $\sigma_1(x, y) = (\sigma(x), y)$, and the map $f_0 = \ctrl{F}{[w]_0}$. Let
\[ G = \langle \sigma_1, \ctrl{F}{[w]_0} \rangle. \]
Note that $f_i = \ctrl{F}{[u]_{-i}} = f_0^{\sigma_1^i} \in G$.

Let $F'$ be any finite set of even permutations of sets of the form $\{0,1\}^k$ such that every even permutation of $\{0,1\}^m$ for any large enough $m$ can be decomposed into application of permutations in $F'$ in contiguous subsequences $\{i, i+1, ..., i+k-1\}$ of the indices $\{0,1,...,m-1\}$. It is well-known that there exist such universal reversible gate sets. Note that $\{F\}$ need not be such a set: we may need to use more than $m$ coordinates to build permutations of $\{0,1\}^m$ using translates of $F$.

For any $i$, since $w$ is unbordered and of length $\ell$, the maps $f_i, f_{i+1}, ..., f_{i+\ell-n}$ compose in the natural way, just like translates of $F$ inside $\{0,1\}^\ell$. By universality of $F$, as long as $\ell$ is large enough, the maps $\ctrl{f'}{[w]_i}$, $f' \in F'$, are generated. By the universality property of $F'$, we have $\ctrl{\pi}{[w]_i} \in G$ for all $\pi \in\Alt(\{0,1\}^\ell)$.

Now, we need to show that also $\ctrl{\pi}{[ww]_i} \in G$. For this, pick a large \emph{mutually unbordered} set $U \subset \{0,1\}^\ell$, i.e.\ any set such that $u_1, u_2 \in U$ have no nontrivial overlaps. For example we can pick $U = 0^{\ell-k-2} 1 \{0,1\}^k 1$
for any $k$ such that $k < \frac{\ell-4}{2}$. By the above, we can perform any even permutation of $U$ under occurrences of $w$. For two permutations $\pi_1, \pi_2 \in \Alt(\{0,1\}^\ell)$, with supports contained in $U$, a direct computation shows
\[ [ \ctrl{\pi_1}{[w]_i}, \ctrl{\pi_2}{[w]_{i + \ell}} ] = \ctrl{[\pi_1, \pi_2]}{[ww]_i}, \]
so for $|U| \geq 5$ ($\ell$ has to be large enough for this) we have
$\ctrl{\pi}{[ww]_i} \in G$
for all $\pi \in \Alt(\{0,1\}^\ell)$ with support contained in $U$.

For two permutations $\pi_1, \pi_2 \in \Alt(\{0,1\}^\ell)$, a direct computation shows
\[ (\ctrl{\pi_1}{[ww]_i})^{\ctrl{\pi_2}{[w]_i}} = \ctrl{(\pi_1^{\pi_2})}{[ww]_i} \]
so, since $\Alt(\{0,1\}^\ell)$ is simple (supposing $\ell \geq 3$), $G$ in fact contains $\ctrl{\pi}{[ww]_i} \in G$ for all $\pi \in \{0,1\}^\ell$. This concludes the proof since $G$ is clearly a quotient of $G' = \Z * \Z_2$, as it was generated by an RCA of infinite order and an involution.

Let us then show the claim for $G' = \Z_2 * \Z_2 * \Z_2$. For this, pick $B = \{0,1\}$ and add a third component $B' = \{0,1\}$ on top, so the alphabet becomes $A = B' \times B \times C$, $m = 8$. Thinking of $x \in (B' \times B \times C)^\Z$ as having three binary tracks, and writing $\sigma_0$ and $\sigma_1$ for the shifts on the first two tracks, it is easy to see that $\sigma_0^{-1} \times \sigma_1$ is the composition of two involutions, say $\sigma_0^{-1} \times \sigma_1 = a \circ b$.

In the proof of universality in \cite{Sa18a}, the shift on the first ($B$-)track is only used to construct the generators of an arbitrary f.g.\ group, but total sum of shifts is $0$ in the elements giving the embedding. Thus, $G = \langle a, b, f_0\rangle$, where $f_0$ is as above but ignores the $B'$-track, is clearly f.g.-universal, and a quotient of $G'$.
\end{proof}

\bibliographystyle{plain}
\bibliography{../../../bib/bib}{}

\end{document}